\documentclass[11pt]{article}
\usepackage{amsmath, amsthm, amscd, amsfonts, amssymb, graphicx,enumerate}

\theoremstyle{definition}
\newtheorem{theorem}{Theorem}[section]
\newtheorem{definition}[theorem]{Definition}
\newtheorem{lemma}[theorem]{Lemma}
\newtheorem{corollary}[theorem]{Corollay}
\newtheorem{proposition}[theorem]{Proposition}
\newtheorem{remark}[theorem]{Remark}

\begin{document}

\title{On the operator Acz\'{e}l inequality and its reverse}
\author{Shigeru Furuichi$^1$\footnote{E-mail:furuichi@chs.nihon-u.ac.jp},  Mohammad Reza Jabbarzadeh$^2$\footnote{E-mail:mjabbar@tabrizu.ac.ir},  and  Venus Kaleibary$^2$\footnote{E-mail:v.kaleibary@gmail.com}\\
$^1${\small Department of Information Science, College of Humanities and Sciences, Nihon University,}\\
{\small 3-25-40, Sakurajyousui, Setagaya-ku, Tokyo, 156-8550, Japan}\\
$^2${\small Faculty of Mathematical Sciences, University of Tabriz
5166615648, Tabriz, Iran}}
\date{}
\maketitle

{\bf Abstract.}
In this paper, we present some operator and eigenvalue inequalities involving  operator monotone, doubly concave and doubly convex functions. 
 These inequalities provide some variants of  operator Acz\'{e}l inequality and its reverse via generalized  Kantorovich constant.
\vspace{3mm}

\vspace{3mm}

\vspace{3mm}


\section{Introduction}
Let $ B(\mathcal{H})$ denote the $C^*$-algebra of all bounded linear operators on a Hilbert space $( \mathcal{H}, \langle \cdot , \cdot \rangle)$. An operator $A \in B(\mathcal{H})$ is called \textit{positive} if $\langle Ax , x \rangle \geq 0$  for every $x \in \mathcal{H}$ and then we write $A \geq 0$. For self-adjoint operators $A,B \in  B(\mathcal{H})$, we say $A \leq B$ if $B - A \geq 0$. Also, we say $A$ is \textit{strictly positive} and we write $A > 0$, if $\langle Ax , x \rangle > 0$  for every $x \in \mathcal{H}$. Let $f$ be a continuous real function on $(0,\infty)$. Then $f$ is said to be \textit{operator monotone}
(more precisely, operator monotone increasing) if $A \geq B$ implies $f(A) \geq f(B)$ for strictly positive operators $A, B $, and \textit{operator monotone decreasing} if $-f$ is operator monotone or
$A \geq B$ implies $f(A) \leq f(B)$.
Also, $f$ is said to be \textit{operator convex} if $f((1-\alpha) A+ \alpha  B) \leq (1-\alpha) f(A)+ \alpha f(B)$
for all strictly positive operators $A,B$ and $ \alpha \in [0, 1]$, and \textit{operator concave} if $-f$ is operator convex.

In 1956, Acz\'{e}l \cite{Ac} proved that if $a_i, b_i (1 \leq i \leq n)$ are positive real numbers such that
$a_1^2 - \sum_{i=2}^n a_i^2 >0 $ and $b_1^2 - \sum_{i=2}^n b_i^2 >0 $, then
\begin{align*}
\left(a_1 b_1 - \sum_{i=2}^n a_i b_i \right)^2 \geq  \left(a_1^2 - \sum_{i=2}^n a_i^2 \right) \left(b_1^2 - \sum_{i=2}^n b_i^2 \right).
\end{align*}
Acz\'{e}l's inequality has important applications in the theory of functional equations in
non-Euclidean geometry \cite{Ac, TW2014} and considerable attention has been given to this inequality involving its
generalizations, variations and applications. See \cite{D1994, P1995} and references therein. 
Popoviciu \cite{P1995} first presented an exponential extension of Acz\'{e}l's
inequality as follows:

\begin{theorem}{\bf (\cite{P1995})}\label{thmx01}
Let $p>1, q>1, \frac{1}{p}+\frac{1}{q}=1$, 
$a_1^p - \sum_{i=2}^n a_i^p >0 $, and $b_1^q - \sum_{i=2}^n b_i^q >0 $. Then
\begin{align*}
 a_1 b_1 - \sum_{i=2}^n a_i b_i \geq  \left(a_1^p - \sum_{i=2}^n a_i^p \right)^\frac{1}{p} \left(b_1^q - \sum_{i=2}^n b_i^q \right)^\frac{1}{q}.
\end{align*}
\end{theorem}
 Acz\'{e}l's  and Popoviciu's inequalities were
sharpened and a variant of Acz\'{e}l's inequality in inner product spaces was given by
Dragomir \cite{D1994}. Recently, 
Moslehian in \cite{Mos2011} proved an operator version of the classical Acz\'{e}l inequality involving $\alpha$-geometric mean 
\begin{equation}\label{def_alpha_geo_mean}
A \sharp_\alpha B = A^{1/2} ( A^{-1/2} B A^{-1/2})^\alpha A^{1/2}
\end{equation}
 for $A>0$, $B \geq 0$ and $\alpha \in [0,1]$, in the following form:  

\begin{theorem}  {\bf (\cite{Mos2011})} \label{thmx1.1}
Let $J$ be an interval of $(0,\infty)$, let $f : J  \to (0, \infty)$ be  operator decreasing and operator concave on $J$, $\frac{1}{p}+ \frac{1}{q}=1, p,q>1 $ and let $A,B \in \mathbb{B}(\mathcal{H})$ be positive invertible operators with spectra contained in $J$. Then
\begin{eqnarray} 
&&f (A^p \sharp_\frac{1}{q} B^q)  \geq f(A^p)\sharp_\frac{1}{q} f(B^q), \label{thmx1.1_ineq01}\\ 
&& \langle f(A^p \sharp_\frac{1}{q}B^q) x, x \rangle \geq  \langle f(A^p) x, x \rangle ^\frac{1}{p} \langle f(B^q) x, x \rangle^\frac{1}{q}. \label{thmx1.1_ineq02}
\end{eqnarray}
for any vector $x \in \mathcal{H}$.
\end{theorem}\label{thmx 1.1}
After that, Kaleibary and Furuichi in \cite{KF2018} provided a reverse of operator Acz\'{e}l inequality using Kantorovich constant as follows.
\begin{theorem}{\bf (\cite[Theorem 1]{KF2018})}\label{thmx1.2}
 Let $ g $ be a non-negative operator monotone decreasing function on $(0, \infty)$, $\frac{1}{p}+ \frac{1}{q}=1, p,q>1$,  and $0 < s A^p \leq B^q  \leq t A^p$ for some scalars $0 < s \leq t$. Then, for all $x \in \mathcal{H}$
\begin{align} 
&g (A^p \sharp_\frac{1}{q} B^q)  \leq \max \lbrace  K(s)^R, K(t)^R \rbrace g(A^p)\sharp_\frac{1}{q} g(B^q), \label{thmx 1.2_ineq01}\\ 
& \langle g(A^p \sharp_\frac{1}{q}B^q) x , x \rangle \leq \max \lbrace  K(s)^R, K(t)^R \rbrace \langle g(A^p) x , x \rangle ^\frac{1}{p} \langle g(B^q) x , x \rangle^\frac{1}{q}, \label{thmx 1.2_ineq02}
\end{align}
where $R= \max \left\{ \frac{1}{p}, \frac{1}{q} \right\} $, and  $K(h) = \dfrac{(h+1)^2}{4h}, h>0$ is the Kantorovich constant.
\end{theorem}
In this paper, we first investigate some operator and eigenvalue inequalities involving  operator monotone, doubly concave and doubly convex functions. Then  we provide another type of operator Acz\'{e}l inequalities along with their reverse using the obtained results.
Since  for a nonnegative continuous function $f$ defined on $(0,\infty)$ the operator concavity is equivalent to the operator monotonicity, the assumptions on $f$  in Theorem \ref{thmx1.1}  seem to be slightly strong, in the special case of $J=(0,\infty)$. Hence, we aim to prove a variant of Theorem \ref{thmx1.1} for the reduced condition such as a non-negative operator monotone function $f$. As an application,  we present a counterpart of the classical  Acz\'{e}l inequality stated in Theorem \ref{thmx01}. These results are organized in Sections 2. Section 3 is devoted to study of Acz\'{e}l type inequality involving doubly concave functions.  In Sections 4, we show several eigenvalue inequalities involving  $\alpha$-geometric mean and doubly convex  functions. The obtained eigenvalue inequalities allow us  to study the reverse of operator Acz\'{e}l inequality  via the generalized Kantorovich
constant $K(w,\alpha)$.
The assumptions of doubly convexity (concavity) will be discussed in more details in later sections. 

\section{A variant of operator Acz\'{e}l inequaltiy}
In this section, we  present a variant of operator Acz\'{e}l inequaltiy by using several reverse Young's inequalities.
Let $A$ and $B$, be strictly positive operators. For each $\alpha \in [0, 1]$ the \textit{$\alpha$-arithmetic} mean is defined as
$A\bigtriangledown_\alpha B := (1-\alpha)A + \alpha B$ and the \textit{$\alpha$-geometric} mean is defined in \eqref{def_alpha_geo_mean}.
Clearly if $AB = BA$, then $A\sharp_\alpha B = A^{1-\alpha}B^\alpha$.
Basic properties of the arithmetic and geometric means can be found in \cite{FMPS2005}.
It is well-known as the Young inequality
\begin{align*}
A \sharp_\alpha B \leq A\nabla_{\alpha}B.
\end{align*}
The research on the Young inequality is interesting and there are several multiplicative and additive reverses of this inequality. We give here some reverse inequalities for the  operators with the sandwich consition $ 0 < s A \leq  B \leq t A$. 
\begin{lemma}{\bf (\cite[Lemma 2]{KF2018})} \label{lemma1.5}
 Let $ 0 < s A \leq  B \leq t A$ for some scalars $0 < s \leq t $ and $\alpha \in [0, 1]$. Then 
\begin{align}  \label{lemma1.5_ineq01}
A\nabla_{\alpha} B \leq  \max \lbrace  K(s)^R, K(t)^R \rbrace ( A \sharp_\alpha B ),
\end{align}
where $K(\cdot)$ is the Kantorovich constant  defined in Theorem \ref{thmx1.2} and  $R= \max \lbrace \alpha, 1-\alpha \rbrace $.
\end{lemma}
The function $K(\cdot)$ is decreasing on $(0, 1)$ and increasing on $[1, \infty)$, $K(t)=K(\frac{1}{t})$, and $K(t)\geq 1$ for every $t>0$ \cite{FMPS2005}.

\begin{lemma}{\bf (\cite[Proposition 1]{KF2018})} \label{lemma1.1}
 Let $ g $ be a non-negative operator monotone decreasing function on $(0, \infty)$ and $0 < s A \leq  B  \leq t A$ for some scalars $0 < s \leq t$. Then, for all $\alpha \in [0 , 1]$
\begin{align}  \label{lemma1.1_ineq01}
\frac{1}{c} g( A \sharp_\alpha B)  \leq g( c( A \sharp_\alpha B) ) \leq  g(A)\sharp_\alpha g(B),
\end{align}
where $c= \max \lbrace  K^R(s), K^R(t) \rbrace$  with the Kantorovich constant $K(\cdot)$ and $R= \max \lbrace \alpha, 1-\alpha \rbrace $.
\end{lemma}

\begin{lemma}\label{lemma1.2}
Let $f$ be a non-negative operator monotone function on $(0, \infty)$ and $0 < s A \leq  B  \leq t A$ for some constants $0 < s \leq t$. Then, for all $\alpha \in [0 , 1]$ we have
$$
c f( A\sharp_\alpha B)\geq f(c (A\sharp_\alpha B)) \geq  f(A)\sharp_\alpha f(B),
$$
where $c= \max \lbrace  K^R(s), K^R(t) \rbrace$ with the Kantorovich constant $K(\cdot)$ and $R= \max \lbrace \alpha, 1-\alpha \rbrace $. 
\end{lemma}
\begin{proof}
First note that since $f$ is analytic on $(0, \infty)$, we may assume that $f (x) > 0$
for all $x > 0$; otherwise $f$ is identically zero.
Also, since $f$ is  operator monotone function on $(0, \infty)$, so $\dfrac{1}{f}$ is a non-negative operator monotone decreasing function on $(0, \infty)$. By Applying  Lemma \ref{lemma1.1}  for $g = \dfrac{1}{f}$ we have
\begin{align*}  
\frac{1}{c} f( A \sharp_\alpha B)^{-1}  \leq f( c( A \sharp_\alpha B) )^{-1} \leq  f(A)^{-1}\sharp_\alpha f(B)^{-1}= \left( f(A)\sharp_\alpha f(B)\right) ^{-1}.
\end{align*}
Reversing the all sides gives the desired inequality. 
\end{proof}

\begin{theorem} \label{theorem1.2}
 Let $ f $ be a non-negative operator monotone function on $(0, \infty)$, $\frac{1}{p}+ \frac{1}{q}=1, p,q>1$,  and $0 < s A^p \leq B^q  \leq t A^p$ for some scalars $0 < s \leq t$. Then we have
\begin{equation}\label{theorem1.2_ineq01}
f\left(A^p\sharp_{1/q} B^q\right) \geq  \frac{1}{c} f(A^p)\sharp_{1/q} f(B^q).
\end{equation}
and
\begin{equation}\label{theorem1.2_ineq02}
\langle f\left(A^p\sharp_{1/q} B^q\right)x,x\rangle  \geq \frac{1}{c} \langle f(A^p)x,x\rangle^{1/p} \langle f(B^q)x,x\rangle^{1/q} 
\end{equation}
for all $x \in \mathcal{H}$. Where $c= \max \lbrace  K^R(s), K^R(t) \rbrace$ with the Kantorovich constant $K(\cdot)$ and $R= \max \lbrace 1/p, 1/q \rbrace$.
\end{theorem}
\begin{proof}

Putting $A:=A^p$, $B:=B^q$ and $\alpha:=1/q$ in Lemma \ref{lemma1.2}, we have \eqref{theorem1.2_ineq01}. Also, using the first inequality of Lemma \ref{lemma1.2}  with $A:=A^p$, $B:=B^q$ and $\alpha:=1/q$, we have
\begin{eqnarray*}
 c \langle f\left(A^p\sharp_{1/q} B^q\right)x,x\rangle &\geq&  \langle f\left(c (A^p\sharp_{1/q} B^q)\right)x,x\rangle\\
&\geq&   \langle f(A^p\nabla_{1/q} B^q) x,x\rangle \hspace*{1cm}(\text{op. monotonicity of}\,\, f \;\text{with} \,\,\eqref{lemma1.5_ineq01})\\
&\geq& \langle \big(\frac{1}{p}f(A^p)+\frac{1}{q}f(B^q)\big)x,x \rangle \hspace*{1cm} (\text{op. concavity of}\,\, f)\\
&=& \frac{1}{p} \langle f(A^p)x,x\rangle + \frac{1}{q} \langle f(A^q)x,x\rangle \\
&\geq& \langle f(A^p)x,x\rangle^{1/p} \langle f(B^q)x,x\rangle^{1/q} \hspace*{1cm} (\text{AM-AG inequality})
\end{eqnarray*}
which implies \eqref{theorem1.2_ineq02}.
\end{proof}

\begin{remark}
\begin{itemize}
\item[(a)]
The constant $c$ in Theorem \ref{theorem1.2} can be replaced by the constant $\acute{c}:= \max \{S(s),S(t)\}$ where  $S(x):=\dfrac{x^{\frac{1}{x-1}}}{e\log x^{\frac{1}{x-1}}}$ for $x>0$ with $x \neq 1$ is the so-called Specht ratio. See  \cite[Theorem 1]{GK2016}. 
In addition, for $\alpha \in (0,1)$ we have no ordering between the estimates $K^R(h)$, $R= \max \lbrace \alpha, 1-\alpha \rbrace$ and $S(h)$ for $h>0$ with $h \neq 1$ in general. Becasue we have numerical examples such that $K^{0.6}(0.01)-S(0.01) \simeq -1.30357$ and $K^{0.6}(5.0)-S(5.0) \simeq 0.0556589$.
For $\alpha = 1$, although $K(h) \geq S(h)$ for $h>0$ with $h \neq 1$,  it cannot be satisfied in the condition $\alpha:=1/q$  with $1/q+1/p=1$ of Theorem \ref{theorem1.2}.

\item[(b)] 
Our results given in both \eqref{theorem1.2_ineq01} and  \eqref{theorem1.2_ineq02} are weaker than ones in both \eqref{thmx1.1_ineq01} and \eqref{thmx1.1_ineq02}, since $K(h) \geq 1$  (and also $S(h)\geq 1$) for $h>0$ with $h \neq 1$. But our assumption for the function $f$ in Theorem \ref{theorem1.2} is better than one for the function $f$ in Theorem \ref{thmx1.1}.
\end{itemize}
\end{remark}
\begin{corollary}\label{cor1.1}
Let $1/p+1/q=1$ with $p,q >1$. For commuting positive invertible operators $A$ and $B$ with spectra contained in $(1,\infty)$ such that $sA^p\leq B^q \leq tA^p$ for some scalars $0 < s \leq t$, we have for any unit vector $x\in \mathcal{H}$,
\begin{equation}\label{cor1.1_ineq01}
\|(AB)^{1/2} x\|^2-1 \geq \frac{1}{c}\left(\|A^{p/2}x\|^2-1\right)^{1/p} \left(\|B^{q/2}x\|^2-1\right)^{1/q}, 
\end{equation}
where the constant $c$ is given in Theorem \ref{theorem1.2}.
\end{corollary}

\begin{proof}
Taking $f(t)= t-1$ on $(1,\infty)$ in Theorem \ref{theorem1.2}, we get the desiered result.
\end{proof}

From \eqref{theorem1.2_ineq01}, we also have the following corollaries.
\begin{corollary}\label{cor1.2}
Let $1/p+1/q=1$ with $p,q >1$ and $f$ be an operator monotone function on $(0,\infty)$. For  commuting positive invertible operators $A$ and $B$ such that $sA^p\leq B^q \leq tA^p$ for some scalars $0 < s \leq t$, we have
\begin{equation}\label{cor1.2_ineq01}
f(AB) \geq \frac{1}{c}   \; f(A^p)^{1/p} f(B^q)^{1/q},
\end{equation}
where the constant $c$ is given in Theorem \ref{theorem1.2}.
\end{corollary}

\begin{corollary}\label{cor1.3}
Let $1/p+1/q=1$ with $p,q >1$ and $f$ be a non-negative increasing function on $(0,\infty)$ and $a_i,b_i$ be positive numbers such that $0<sa_i^p \leq b_i^q \leq ta_i^p$ for some scalars $0 < s \leq t$. Then we have
\begin{equation}\label{cor1.3_ineq01}
\sum_{i=1}^n f(a_ib_i) \geq \frac{1}{c}  \left(\sum_{i=1}^n f(a_i^p) \right)^{1/p} \left(\sum_{i=1}^n  f(b_i^q) \right)^{1/q} 
\end{equation}
where the constant $c$ is given in Theorem \ref{theorem1.2}.
\end{corollary}

 The following result provides a counterpart of Theorem \ref{thmx01}. 
\begin{corollary}\label{cor1.4}
Let $1/p+1/q=1$ with $p,q >1$.
For positive numbers $x_i$ and $y_i$ such that $\sum_{i=2}^nx_i^p \geq x_1^p$, $\sum_{i=2}^ny_i^q \geq y_1^q$, $\sum_{i=2}^nx_iy_i \geq x_1y_1$ and $0<s \left(\dfrac{x_i}{x_1}\right)^p  \leq \left(\dfrac{y_i}{y_1}\right)^q \leq t\left(\dfrac{x_i}{x_1}\right)^p$ for some scalars $0 < s \leq t$. Then we have
$$
\sum_{i=2}^nx_iy_i -x_1y_1 \geq \frac{1}{c}\left(\sum_{i=2}^nx_i^p-x_1^p\right)^{1/p}\left(\sum_{i=2}^ny_i^q-y_1^q\right)^{1/q}
$$
where the constant $c$ is given in Theorem \ref{theorem1.2}.
\end{corollary}

\begin{proof}
Firstly we note that the inequality \eqref{cor1.3_ineq01}  is true  for any $n\in\mathbb{N}$, as $i=1,\cdots,n-1$ so that we may relabel as it is true for $i=2,\cdots,n$.
Take a function $f(t):=t-\frac{1}{n-1}$, ($n \geq 2$) on $(\frac{1}{n-1}, \infty)$ in Corollary \ref{cor1.3}, then $f(t)$ is non-negative and monotone increasing on $(1, \infty)$. Then we obtain the inequality:
\begin{align}\label{cor1.4_ineq01}
\sum_{i=2}^na_ib_i -1 \geq \frac{1}{c}\left(\sum_{i=2}^na_i^p-1\right)^{1/p}\left(\sum_{i=2}^nb_i^q-1\right)^{1/q}.
\end{align}
Let $x_1,y_1 >0$. Putting $a_i :=\frac{x_i}{x_1}$ and $b_i :=\frac{y_i}{y_1}$ for positive numbers $x_i$ and $y_i$ for $i=2,\cdots,n$ in the above, we obtain
$$
\sum_{i=2}^nx_iy_i -x_1y_1 \geq \frac{1}{c}\left(\sum_{i=2}^nx_i^p-x_1^p\right)^{1/p}\left(\sum_{i=2}^ny_i^q-y_1^q\right)^{1/q},
$$
under the assumptions $\sum_{i=2}^nx_i^p \geq x_1^p$, $\sum_{i=2}^ny_i^q \geq y_1^q$ and $\sum_{i=2}^nx_iy_i \geq x_1y_1$.
\end{proof}

\section{Acz\'{e}l inequalities with the generalized Kantorovich constant for doubly concave function}
In the next we study an analogous of Theorem \ref{theorem1.2}, with the generalized Kantorovich constant $K(w, \alpha)$. For this purpose,  the assumption of doubly concavity of $f(t)$ is needed. 
\begin{definition}
A  non-negative continuous function $f(t)$  defined on a positive interval $J \subset [0,\infty)$, is said to be doubly concave if:
\begin{enumerate}
 \item $f(t)$ is concave in the usual sense;
 \item $f(t)$ is geometrically concave, i.e., $ g(x^\alpha y^{1-\alpha}) \geq g(x)^\alpha g(y)^{1-\alpha}$  for all $x, y \in I$, and $\alpha \in [0, 1]$.
\end{enumerate}
\end{definition}
\par
\noindent
If $f(t)$ and $g(t)$ are doubly concave on $J$, then so is their geometric mean $f(t)^\alpha g(t)^{1-\alpha}$
for $\alpha \in [0, 1]$ and their minimum $\min\lbrace f(t), g(t) \rbrace$. These properties say that
 there are a lot of doubly concave functions.  
 The most important examples of doubly concave functions on $J= [0,\infty)$ are $t \mapsto t^p$ with exponent $p \in [0, 1]$. Other simple examples are $t \mapsto t/(t+1)$,
$t \mapsto t/ \sqrt{t+1}$ and $t \mapsto 1- e^{-t}$. On $J= [1,\infty)$, the functions $\log t$ and $(t-1)^p,  p\in [0, 1]$, and on $J= [0,1]$, the function $-t \log t$ are also doubly concave. For more examples see \cite{BH2014}.

Now, we are ready to give a result  via the constant $K(w,a)$ occuring in the following lemma.
\begin{lemma}{\bf (\cite[Lemma 8]{BLFS2009})} \label{lemma 1.3}
Let $A, B > 0$ with $0 <s A \leq B \leq t A$ for some scalars $0< s \leq t$ with $w = t/s$. Then, for all vectors $x$
and all $\alpha \in [0 , 1]$
\begin{align*}
\langle A\sharp_\alpha B x , x\rangle \leq \langle A x, x\rangle^{1-\alpha} \langle B x, x\rangle^{\alpha}\leq  K^{-1}(w, \alpha) \langle A\sharp_\alpha B x, x\rangle,
\end{align*}
where $K(w, \alpha)$ is the generalized Kantorovich constant defined for $w > 0$ by:
\begin{align}\label{lemma1.3_ineq01}
K(w, \alpha):= \dfrac{\left(w^\alpha - w\right)}{(\alpha-1)(w-1)} \left( \dfrac{\alpha - 1}{\alpha}\dfrac{w^\alpha-1}{w^\alpha - w} \right)^\alpha.
\end{align}
\end{lemma}
It is known that $K(w, \alpha) \in (0, 1]$ for $\alpha \in [0 , 1]$. See \cite{FMPS2005} for some important properties of $K(w, \alpha)$.

\begin{lemma}{\bf (\cite[Theorem 1]{GK2017})} \label{lemma1.4}
 Let $f$ be an increasing doubly concave function on $[0, \infty)$ and $0 <s A \leq B \leq t A$ for some scalars $0 < s \leq t$ with $w = t/s$. Then for all $\alpha \in [0 , 1]$ and $ k=1,2,\cdots,n$,
\begin{align*} \label{}
K^{-1}(w,\alpha)\lambda_k \left(f( A \sharp_\alpha B )\right) \geq \lambda_k \left( f\left(   K^{-1}(w,\alpha)(A \sharp_\alpha B) \right)\right)\geq \lambda_k \left(f(A) \sharp_\alpha f(B)\right).
\end{align*}
where $K(w, \alpha)$ is the generalized Kantorovich constant defined as \eqref{lemma1.3_ineq01} .
\end{lemma}
This statement is equivalent to the existence of a unitary operator $U$ satisfying
the following inequality:
\begin{align} \label{lemma1.4_ineq01}
f( A \sharp_\alpha B ) \geq K(w,\alpha)  U \left(f(A) \sharp_\alpha f(B)\right) U^*.
\end{align}
Also, from the proof of \cite[Theorem 1]{GK2017} it is inferred that the right hand side inequality holds for an increasing geometrically concave function too.
\begin{theorem} \label{theorem1.3}
 Let $f$ be an increasing doubly concave function on $[0, \infty)$,  $\frac{1}{p}+ \frac{1}{q}=1, p,q>1$, and $0 <s A^p \leq B^q \leq t A^p$ for some scalars $0 < s \leq t$ with $w = t/s$. Then, there is a unitary operator $U$ such that
\begin{equation}\label{theorem1.3_ineq01}
f\left(A^p\sharp_{1/q} B^q\right) \geq  K(w,1/q)  \;U \left( f(A^p)\sharp_{1/q} f(B^q) \right) U^*,
\end{equation}
where $K(w, \alpha)$ is the generalized Kantorovich constant defined as \eqref{lemma1.3_ineq01}. In addition, if $f$ is an operator monotone function and $s \leq 1 \leq t$, then for all $x \in \mathcal{H}$
\begin{equation}\label{theorem1.3_ineq02}
\langle f\left(A^p\sharp_{1/q} B^q\right) Ux, Ux\rangle  \geq K^2(w, 1/q) \langle f(A^p)x,x\rangle^{1/p} \langle f(B^q)x,x\rangle^{1/q}. 
\end{equation}
\end{theorem}
\begin{proof}
Putting $A:=A^p$, $B:=B^q$ and $\alpha :=1/q$ in Lemma \ref{lemma1.4}, we have \eqref{theorem1.3_ineq01}.
For the inequality \eqref{theorem1.3_ineq02}, we first note that since $f$ is an operator monotone function, $sA^p\leq B^q \leq tA^p$ implies $f(sA^p)\leq f(B^q) \leq f(tA^p)$. Since  $s \leq 1 \leq t$, by the cocavity of $f$ we have $s f(A^p)\leq f(B^q) \leq t f(A^p)$ and so the condition number of operators $f(A^p)$ and $f(B^q)$ is also $w$. Now we have
\begin{eqnarray*}
\langle U^* f\left(A^p\sharp_{1/q} B^q\right) Ux,x\rangle 
&\geq& K(w,1/q)  \langle \left( f(A^p)\sharp_{1/q} f(B^q) \right)x,x\rangle  \quad (\text{by \eqref{theorem1.3_ineq01}})\\
&\geq&  K^2(w,1/q)   \langle f(A^p)x,x\rangle^{1/p} \langle f(B^q)x,x\rangle^{1/q} \quad (\text{Lemma \ref{lemma 1.3}}).
\end{eqnarray*}
\end{proof}

\begin{corollary}\label{cor2.1}
Let $1/p+1/q=1$ with $p,q >1$. For commuting positive invertible operators $A$ and $B$ with spectra contained in $(1,\infty)$ such that $sA^p\leq B^q \leq tA^p$ for $0 < s < 1 < t$, there is a unitary operator $U$ that for any unit vector $x\in \mathcal{H}$,
\begin{equation}\label{cor2.1_ineq01}
\| U^* (AB)^{1/2}U x\|^2-1 \geq K^2(w,1/q)  \left(\|A^{p/2}x\|^2-1\right)^{1/p} \left(\|B^{q/2}x\|^2-1\right)^{1/q}.
\end{equation}
\end{corollary}
\begin{proof}
Taking $f(t)= t-1$ on $(1,\infty)$ in the inequality \eqref{theorem1.3_ineq02}, we get the desiered result. Note that this function is both  operator monotone and doubly concave function on $(1,\infty)$. So, we have
\begin{eqnarray*}
\langle U^* (A^p\sharp_{1/q} B^q - 1) Ux, x\rangle &=& \langle U^* (AB - 1) Ux,x\rangle = \langle U^* AB U x, x\rangle- \langle  x,x\rangle \\
&=& \| (U^* AB U)^{1/2} x \|^2- 1= \| U^* (AB)^{1/2} U x \|^2- 1. 
\end{eqnarray*}
The right hand side of the inequality is obtained in a similar way.
\end{proof}

\section{Reverse inequalities with generalized Kantorovich constant for doubly convex functions }

Theorem \ref{thmx1.2} provided a reverse of an operator Acz\'{e}l inequality with Kantorovich constant $K(t)$. Also, it has been proved for a non-negative operator decreasing function $g$.  
In this section we are going to present  some another reverse of an operator Acz\'{e}l inequality via generalized Kantorovich constant $K(w, \alpha)$. For this aim we need doubly convex functions.
\begin{definition}A  non-negative continuous function $g(t)$  defined on a positive interval $J \subset [0,\infty)$, is said doubly convex if:
\begin{enumerate}
 \item $g(t)$ is convex in the usual sense;
 \item $g(t)$ is geometrically convex, i.e., $ g(x^\alpha y^{1-\alpha}) \leq g(x)^\alpha g(y)^{1-\alpha}$  for all $x, y \in J$, and $\alpha \in [0, 1]$.
\end{enumerate}
\end{definition}
\par 
\noindent
Given real numbers $c_i \geq 0$ and $\alpha_i \in (-\infty, 0] \cup [1,\infty), i = 1, \ldots , n$, the
function $g(t) :=\Sigma_{i=1}^n c_i t^{\alpha_i }$ is doubly convex on $(0,\infty)$. See \cite{BH2014}.
 
\begin{lemma} {\bf (\cite[p. 58]{B1997}
(The Minimax Principle)} 
Let $A$ be a Hermitian operator on $\mathcal{H}$. Then
\begin{align*}
\lambda_k (A) 
&= \min_{ \dim \mathcal{F} = n-k+1 }  \max \big\lbrace  \langle Ah, h\rangle ; \;  h \in \mathcal{F},  \; \| h\|=1 \big\rbrace,\nonumber
\end{align*}
where $ \mathcal{F} $ is a subspace of $\mathcal{H}$.
\end{lemma}
The following result gives an analogous of  Lemma \ref{lemma1.1} with the constant $K(w,\alpha)$.
\begin{proposition}\label{proposition3.2}
 Let $ g $ be an increasing doubly convex function on $(0, \infty)$ and $A, B$ be  positive definite matrices such that $0 < s g(A) \leq  g(B)  \leq t g(A)$ for some scalars $0 < s \leq t$. Then, for all $\alpha \in [0 , 1]$ and $k=1,2, \ldots,n$
\begin{align}\label{proposition3.2_ineq01}
 \lambda_k \left(g( A \sharp_\alpha B)\right) \leq  K^{-1}(w, \alpha) \lambda_k \left(g(A)\sharp_\alpha g(B) \right),
\end{align}
where $K(w,\alpha)$ is the generalized Kantorovich constant defined as  \eqref{lemma1.3_ineq01}.
\end{proposition}
\begin{proof}
We will use the
following observation which follows from the standard Jensen's inequality: for any
vector $x$ whose norm is less than or equal to one, since $g$ is convex 
$
\langle g(A)x, x \rangle \geq g (\langle Ax, x \rangle)$. 
For any integer $k$ less than or equal to the dimension of the space, we have a subspace $\mathcal{F}$ of dimension $n-k+1$ such that
\begin{align} \label{}
&\lambda_k  (g(A) \sharp_\alpha g(B)) \nonumber\\
&= \max_{x \in \mathcal{F} : \| x\|=1} \langle  g(A) \sharp_\alpha g(B) x , x \rangle \hspace*{1cm}\text{ (minmax principle)}\nonumber\\
& \geq  \max_{x \in \mathcal{F} : \| x\|=1}  K(w,\alpha) \langle g(A) x , x \rangle^{1-\alpha}  \langle  g(B) x , x \rangle^{\alpha} \hspace*{1cm}\text{ (Lemma \ref{lemma 1.3})}\nonumber\\
&=\max_{h \in \mathcal{F} : \| x\|=1}  K(w,\alpha) ( g \langle A x , x \rangle)^{1-\alpha}  (g \langle B x , x \rangle)^{\alpha}\hspace*{1cm}\text{ (convexity of $g$)}\nonumber\\
&\geq \max_{x \in \mathcal{F} : \| x\|=1} K(w,\alpha) g\big( \langle A x , x \rangle^{1-\alpha}  \langle B x , x \rangle^{\alpha} \big) \hspace*{1cm}\text{ (geometrically convexity of $g$)}\nonumber\\
& \geq \max_{x \in \mathcal{F} : \| x\|=1} K(w,\alpha) g \big( \langle A \sharp_\alpha B x , x \rangle \big)  \hspace*{1cm}\text{ (Lemma \ref{lemma 1.3})}\nonumber\\
& = K(w,\alpha) \max_{x \in \mathcal{F} : \| x\|=1}  \langle g(A \sharp_\alpha B) x , x \rangle \hspace*{1cm}\text{ (monotonicity of $g$)}\nonumber\\ 
&\geq K(w,\alpha) \lambda_k (g( A \sharp_\alpha B )), \hspace*{1cm}\text{ (minmax principle)}.\nonumber
\end{align}
\end{proof}

\begin{remark}
We know that the above statement is equivalent to the existence of a unitary operator $U$ satisfying in
the following inequality:
\begin{align} \label{proposition3.2_ineq03}
  g( A \sharp_\alpha B )  \leq K^{-1} (w,\alpha) U (  g(A) \sharp_\alpha g(B) ) U^*.
\end{align}
This result provides a reverse of the inequality \eqref{lemma1.4_ineq01} for doubly convex functions.
\end{remark}

Applying Proposition \ref{proposition3.2} we achieve the following  reverse operator Acz\'{e}l inequality.

\begin{theorem}\label{theorem3.2}
Let $ g $ be an increasing doubly convex function on $(0, \infty)$, $\frac{1}{p}+ \frac{1}{q}=1, p,q>1$ and  $s g(A^p) \leq  g(B^q)  \leq t g(A^p)$ for some scalars $0 < s \leq t$. Then, there is a unitary operator $U$ such that for all  $x \in \mathcal{H}$
\begin{align} 
&g (A^p \sharp_\frac{1}{q} B^q)  \leq K^{-1}(w,1/q)  U \big(g(A^p)\sharp_\frac{1}{q} g(B^q) \big) U^* , \label{theorem3.2_inequality01}\\ 
& \langle g(A^p \sharp_\frac{1}{q}B^q) Ux , Ux \rangle \leq  K^{-1}(w,1/q) \langle g(A^p) x , x \rangle^{\frac{1}{p}} \langle  g(B^q)  x , x \rangle ^{\frac{1}{q}}. \label{theorem3.2_inequality02}
\end{align}
\end{theorem}
\begin{proof}
Letting $\alpha := \frac{1}{q}$ and replacing $A^p$ and $B^q$ with $A$ and $B$ in the inequality \eqref{proposition3.2_ineq03}, we reach the first inequlity. For the second, we have 
\begin{eqnarray*} 
\langle U^* g(A^p \sharp_\frac{1}{q}B^q) U x , x \rangle
&\leq& K^{-1}(w,1/q) \big\langle \big(g(A^p)\sharp_\frac{1}{q} g(B^q) \big) x , x \big\rangle \hspace*{1cm} (\text{by} \,\, \eqref{theorem3.2_inequality01}) \\
&\leq&  K^{-1}(w,1/q) \langle g(A^p) x , x \rangle^{\frac{1}{p}} \langle g(B^q)  x , x \rangle ^{\frac{1}{q}},\hspace*{1cm}\text{(Lemma \ref{lemma 1.3})}.
\end{eqnarray*}
\end{proof}
\begin{remark}
Theorem \ref{theorem3.2} is a conjugate of Theorem \ref{theorem1.3}, which gives a revese operator Acz\'{e}l inequality relevant to the generalized Kantorovich constant. 
\end{remark}
In the following, we will present another reverse of Acz\'{e}l inequality involving decreasing geometrically convex functions.
Note that $x^p$ for $p<0$ on $(0,\infty)$ and $\csc(x)$ on 
$(0,\frac{\pi}{2})$ are examples of decreasing geometrically convex functions.   In what follows, the capital letters $A, B$ means $n \times n$ matrices or bounded linear operators on an $n$-dimentional complex Hilbert space $\mathcal{H}$.
\begin{proposition} \label{proposition3.1}
 Let $g$ be a decreasing geometrically convex function on $(0, \infty)$ and $0 <s A \leq B \leq t A$ for some scalars $0 < s \leq t$ with $w = t/s$. Then for all $\alpha \in [0 , 1]$ and $ k=1,2,\cdots,n$,
\begin{align}\label{proposition3.1_ineq01} 
\lambda_{k} \left(g\left( K^{-1}(w,\alpha) (A \sharp_\alpha B) \right)\right) \leq \lambda_{k}  \left(g(A) \sharp_\alpha g(B)\right). 
\end{align}
\end{proposition}
\begin{proof}
Since $g$ is an decreasing geometrically convex function $(0, \infty)$, so $f= 1/g$ is an increasing geometrically concave function on  $(0, \infty)$ as follows:
\begin{align*} \label{}
f(x)^{\alpha} f(y)^{1-\alpha}=\dfrac{1}{g(x)^{\alpha} g(y)^{1-\alpha}} \leq   \dfrac{1}{g(x^{\alpha} y^{1-\alpha})} = f(x^{\alpha} y^{1-\alpha}).
\end{align*}
Furthermore, according to Lemma \ref{lemma1.4} for every increasing geometrically concave function $f$ 
\begin{align*} 
\lambda_k  \left(f(A) \sharp_\alpha f(B)\right)  \leq   \lambda_k \left( f\left(   K^{-1}(w,\alpha)(A \sharp_\alpha B) \right)\right).
\end{align*}
Now, by applying this inequality for the function $f= 1/g$ we have
\begin{align}
\lambda_k  \left(g(A)^{-1} \sharp_\alpha g(B)^{-1}\right)  \leq   \lambda_k \left(g\left(K^{-1}(w,\alpha) (A \sharp_\alpha B) \right)^{-1}\right),
\end{align}
where $ k=1,2,\cdots,n$. Thanks to the property  $A^{-1} \sharp_\alpha B^{-1}=(A \sharp_\alpha B)^{-1} $ we can write
\begin{align*} \label{}
\lambda_k  \left((g(A) \sharp_\alpha g(B))^{-1}\right)  \leq \lambda_k \left(g\left(K^{-1}(w,\alpha) (A \sharp_\alpha B) \right)^{-1}\right).
\end{align*}
On the other hand, for every operator $A>0$,  $\lambda_k (A^{-1})=\lambda_{n-k+1}^{-1}(A)$. Hence
\begin{align*} \label{}
\lambda_{n-k+1}^{-1}  \left(g(A) \sharp_\alpha g(B)\right)  \leq  \lambda_{n-k+1}^{-1} \left(g\left(K^{-1}(w,\alpha) (A \sharp_\alpha B)\right) \right).
\end{align*}
This inequality is equevalent to the following one 
\begin{align*} \label{}
\lambda_{j}  \left(g(A) \sharp_\alpha g(B)\right)  \geq  \lambda_{j} \left(g\left(K^{-1}(w,\alpha) (A \sharp_\alpha B) \right)\right),
\end{align*}
for $ j=1,2,\cdots,n$ as desired.
\end{proof}

\begin{theorem}\label{theorem3.1}
Let $ g $ be a decreasing doubly convex function on $(0, \infty)$, $\frac{1}{p}+ \frac{1}{q}=1, p,q>1$ and  $0 < s A^p \leq  B^q \leq t A^p$ for some constants $0 < s \leq t$. Then, there is a unitary operator $U$ such that for all  $x \in \mathcal{H}$
\begin{align} 
&g \left(K(w, \alpha) (A^p \sharp_\frac{1}{q} B^q) \right)  \leq  U \big(g(A^p)\sharp_\frac{1}{q} g(B^q) \big) U^* ,\label{theorem3.1_ineq01}\\ 
& \left\langle g\left(K(w, \alpha)(A^p \sharp_\frac{1}{q}B^q)\right) Ux , Ux \right\rangle \leq \langle  g(A^p) x , x \rangle^{\frac{1}{p}} \langle  g(B^q)  x , x \rangle ^{\frac{1}{q}}. \label{theorem3.1_ineq02}
\end{align}
\end{theorem}
\begin{proof}
The proof is similar to that of Theorem \ref{theorem3.2} by applying Proposition \ref{proposition3.1}.
\end{proof}

\section*{Acknowledgement}
The author (S.F.) was partially supported by JSPS KAKENHI Grant Number 16K05257.

\end{document}